\documentclass[a4paper,10pt]{article}

\usepackage{amsmath, amsthm, amscd, amsfonts, amssymb, graphicx, color}
\usepackage[bookmarksnumbered, colorlinks, plainpages]{hyperref}

\usepackage{enumitem}
\usepackage{url}
\usepackage{multirow}
\usepackage{subfigure}
\usepackage[pass]{geometry}
\usepackage{cite}
\usepackage{cancel}

\usepackage{tikz}
\usetikzlibrary{calc,shapes,decorations.pathreplacing,matrix}

\tikzset{square matrix/.style={
    matrix of nodes,
    column sep=-\pgflinewidth, row sep=-\pgflinewidth,
    nodes={draw,
      minimum height=4.5pt,
      anchor=center,
      text width=4.5pt,
      align=center,
      inner sep=0pt
    },
  },
  square matrix/.default=1.2cm
}

\newtheorem{thm}{Theorem}

\newtheorem{rem}{Remark}
\newtheorem{defn}{Definition}

\newtheorem{cor}{Corollary}
\newtheorem{obs}{Observation}
\newtheorem{conj}{Conjecture}
\newtheorem{prop}{Proposition}

\begin{document}

\title{Quasi-transversal in Latin Squares}

\author{Adel P. Kazemi$^{1,2}$ and Behnaz Pahlavsay$^{1,3}$ \\[1em]
$^1$ Department of Mathematics, University of Mohaghegh Ardabili, \\ P.O.\ Box 5619911367, Ardabil, Iran. \\[1em]
$^2$ Email: adelpkazemi@yahoo.com \\
$^3$ Email: pahlavsayb@yahoo.com, pahlavsay@uma.ac.ir \\[1em]
}

\maketitle

\begin{abstract}
In this paper, we first present the relation between a transversal in a Latin square with some concepts in its Latin square graph, and give an equivalent condition for a Latin square has an orthogonal mate. The most famous open problem involving Combinatorics is to find maximum number of disjoint transversals in a Latin square. So finding some family of decomposible Latin squares into disjoint transversals is our next aim. In the next section, we give an equivalent statement of a conjecture which has been attributed to Brualdi, Stein and Ryser by the concept of quasi-transversal. Finally, we prove the truth of the Rodney's conjecture for a family of graphs.
\\[0.2em]

\\[0.2em]

\end{abstract}

\section{Introduction}
All graphs considered here are finite, undirected and simple. For standard
graph theory terminology not given here we refer to \cite{West}. Let $
G=(V,E) $ be a graph with \emph{vertex set} $V$ and \emph{edge set} $E$. 
The \emph{open neighborhood} and the
\emph{closed neighborhood} of a vertex $v\in V$ are $N_{G}(v)=\{u\in V\ |\
uv\in E\}$ and $N_{G}[v]=N_{G}(v)\cup \{v\}$, respectively. The \emph{degree}
of a vertex $v$ is also $deg_G(v)=\mid N_{G}(v) \mid $. The \emph{minimum}
and \emph{maximum degree} of $G$ are denoted by $\delta=\delta (G)$ and $%
\Delta=\Delta (G)$, respectively. A graph $G$ is $r$-\emph{regular} if and only if
$\Delta (G)=\delta (G)=r$.

\vspace{0.2cm}
\textbf{Latin Square.} For any $r\leq n$, a  \emph{$n$-by-$r$ Latin rectangle} is a $n$-by-$r$ grid, each entry
of which is a number from the set $[n]=\{1,2,\cdots,n\}$ such that no number
appears twice in any row or column. An $n$-by-$n$ Latin rectangle is called a \emph{Latin square} of order $n$.
If in a Latin square $L$ of order $n$ the $\ell^2$ cells defined by $\ell$ rows and $\ell$ columns form
a Latin square of order $\ell$, it is a Latin subsquare of $L$. It is clear that, if we permute in any way the rows, or the columns, or the symbols, of a Latin square, the result is still a Latin square. We say that two Latin squares $L$ and $L'$ (using the same symbol set) are \emph{isotopic} if there is a triple $(f,g,h)$, where $f$ is a row permutation, $g$ a column permutation, and $h$ a symbol permutation, carrying $L$ to $L'$: this means that if the $(i, j)$ entry of $L$ is $k$, then the $(f(i),g(j))$ entry of $L'$ is $h(k)$. The triple $(f,g,h)$ is called an \emph{isotopy}.

Let $L_1$ be a Latin Square of order $n$ with the symbol set $A_1$, and let $L_2$ 
be a Latin Square of order $n$ with the symbol set $A_2$. If there exists a bijection between $A_1$ and $A_2$, then we say that $L_1$ and $L_2$ are \emph{symbol-isomorphic}. If the bijection is 
$\sigma :A_1 \rightarrow A_2$, we write $L_2=\sigma (L_1)$. 
Assume that a Latin square $L$ of order $n$ can be partitioned to the Latin rectangle $L_1$,$\cdots$,  $L_t$
of sizes $m_1\times n, \cdots, m_t\times n$, respectively, such that the first $m_1$ rows in $L$ is the rows
in $L_1$, the second $m_2$ rows in $L$ is the rows in $L_2$, and so on. Then we simply write $L=(L_1,\cdots,L_t)$. 
 
 A pair of Latin squares $A=[a_{i,j}]$ and $B=[b_{i,j}]$ of order $n$ are 
said to be \emph{orthogonal mates} if the $n^2$ ordered pairs $(a_{i,j},b_{i,j})$
are distinct. 

An \emph{orthogonal array} $OA(n,3)$ of \emph{order} $n$ and \emph{depth} $3$ is a $3$ by
$n^2$ array with the integers $1$ to $n$ as entries, such that for any two
rows of the array, the $n^2$ vertical pairs occurring in these rows are
different. Suppose that we have such an array. Call the rows $r$, $c$,
and $s$, in any order. For any pair $(i,j)$, there is a $k$ such that $r_k = i$,
$c_k = j$. We make a square with entry $s_k$ in the $i$-\emph{th} row and $j$-\emph{th}
column (for all $i$ and $j$). The definition of orthogonal array ensures that
this is a Latin square and we can reverse the procedure. Therefore
the concepts of Latin square and orthogonal array are equivalent. Figure \ref{fig:0} shows an $OA(4,3)$ and corresponding Latin squares. 
%--------------------- Figure 1------------------------
\begin{figure}[h]
    \centering
    \includegraphics[width=95mm, height=47mm]{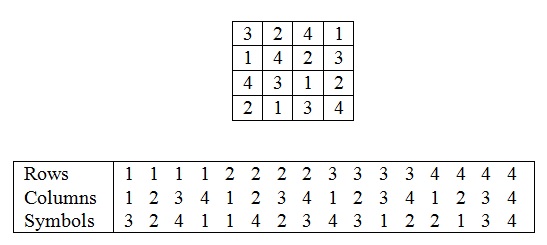}
    \vspace{-0.5cm}
    \caption{A Latin squares of order $4$ and its corresponding $OA(4,3)$}
    \label{fig:0}
\end{figure}
%--------------------- Figure 1------------------------

A \emph{cyclic Latin square} is a Latin square that each column is obtained from the previous one by adding the same difference to each element of the previous column. We know that every finite set $G$ of order 
$n$ with a binary operation defined on $G$ is a semigroup if and only if its operation table is a Latin square of order $n$. We denote this Latin square by $L_G$. 

In contrast to the various constructive procedures for the formation of
sets of mutually orthogonal Latin squares, several conditions that a given Latin square (or set of
squares) be not extendible to a larger set of orthogonal Latin squares are
known. In particular L. Euler \cite{Euler}, E. Maillet \cite{Mailet}, H. B. Mann \cite{Mann} and E. T. Parker \cite{Parker6} and \cite{Parker9} have all given such conditions.
The results of Euler and Maillet concern Latin squares of a particular
kind which the latter author has called $q$-step type. A Latin square of order $mq$ is said to be of \emph{$q$-step type} if it can be represented by a matrix of $q\times q$ blocks $A_{ij}$ as follows:
%--------------------- Figure 1------------------------
\begin{figure}[h]
    \centering
     \includegraphics[width=42mm, height=20mm]{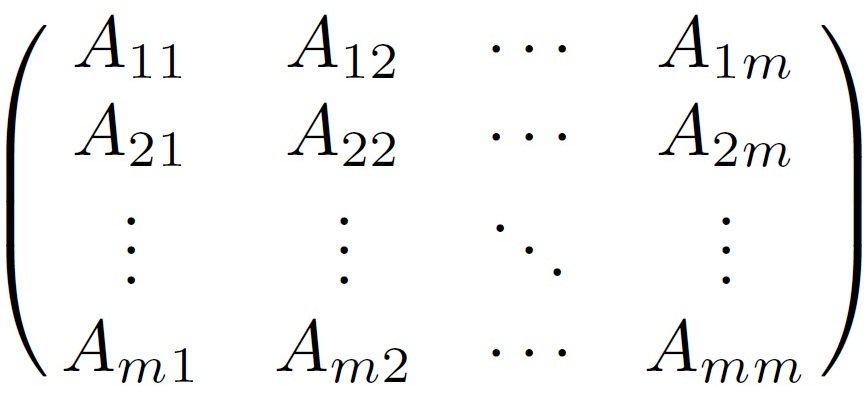}
   % \caption{A Latin squares of order $4$ and its corresponding $OA(4,3)$}
    \label{fig:-1}
\end{figure}
%--------------------- Figure 1------------------------

%\vspace{0.1cm}
where each block $A_{ij}$ is a Latin subsquare of order $q$ and two blocks $A_{ij}$ and $A_{i'j'}$ contain
the same symbols if and only if $i+j\equiv i'+j'$ (mod $m$).
\begin{rem}\label{1-step}
Every cyclic Latin square is a Latin square of $1$-step type.
\end{rem}

A \emph{transversal} in a Latin square of order $n$ is a set of entries which
includes exactly one entry from each row and column and one of each
symbol. The maximum number of disjoint transversals in a Latin square
$L$ is called its \emph{transversal number} and denoted by
$\tau(L)$. Similarly, a transversal in a $n$-by-$r$ Latin rectangle is a set 
of entries which includes exactly one entry from each row, and at most one entry from each column and at most one of each symbol. For any transversal $T$ in $L=(L_1,\cdots,L_t)$ the notation $T=(T_{1},\cdots, T_{t})$ denotes that $T$ is partitioned to the transversals $T_i$ in the disjoint $n$-by-$m_i$ Latin rectangles $L_i$.

If $\tau(L)=n$, then we say it has \emph{a decomposition into disjoint transversals}. In 
\cite{WCS} Wanless, Church and Aldates in a natural way, generalized the concept of transversal to $k$-transversal as following: For any positive integer $k$, a \emph{$k$-transversal} or $k$-\emph{plex} in a Latin square of order $n$ is a set of $nk$ cells, $k$ from each row, $k$ from each column, in which every symbol occurs exactly $k$ times. The maximum number of disjoint $k$-transversals in a Latin square $L$ is called its \emph{$k$-transversal number} and denoted by
$\tau_{k}(L)$. Obviously $\tau_{k}(L)\leq n/k$. If $\tau_{k}(L)=n/k$, then we say it has  \emph{a decomposition into disjoint $k$-transversals}. Obviously, the next remark is true.
\begin{rem}
\label{k,n-k} 
The complement of a $k$-transversal in a Latin square of order $n$ is a $(n-k)$-tarnsversal in the Latin square.
\end{rem}
%-------------------------------------------------------------------------
A \emph {partial transversal} of length $k$ is a set of $k$ entries, each
selected from different rows and columns of a Latin square such that no two entries
contain the same symbol. A partial transversal is \emph{completable} if it is a subset of
some transversal, whereas it is \emph{non-extendible} if it is not contained in any partial
transversal of greater length. Since not all squares of order $n$ have a partial transversal of length $n$ (i.e. a transversal), the best we can hope for is to find one of length $n-1$. Such partial
transversals are called \emph{near-transversals}.

%---------------------------------------------------------------------------------------
\vspace{0.2cm}
\textbf{Latin Square Graph.} For a Latin square $L=(\ell_{i,j})$, we
correspond a Latin square graph $L_3(L,n)$, or briefly $L_3(n)$ if
there is no ambiguity, as following: $V(L_3(L,n))=\{(i,j)| 1\leq i,j\leq n\}$, 
and two distinct vertices $(i,j)$ and $(p,q)$ are adjacent if and only if $i=p$ 
or $j=q$ or $\ell_{ij}=\ell_{pq}$.

It is trivial that $L_3(L,n)=K_{n^2}$ if and only if $n=1,2$. A Latin square graph $L_3(L,n)$ is
a $3(n-1)$-regular graph, and any two different vertices $(i,j)$ and $(i,q)$ have $n$
neighbors in-common, $n-2$ vertices in the row $i$ and two vertices in the columns $j$ and $q$. Similarly, any two different vertices $(i,j)$ and $(p,j)$ have $n$ neighbors in-common. By the corresponding between a Latin square with an unique orthogonal array, it can be easily seen that any two different vertices $(i,j)$ and $(p,q)$ with $\ell_{i,j}=\ell_{p,q}$ have also $n$
neighbors in-common.
%--------------------------------------------------------------------------------------------------

A \emph{proper coloring} of a graph $G$ is a function from the
vertices of the graph to a set of colors such that any two adjacent
vertices have different colors, and the \emph{chromatic number}
$\chi (G)$ of $G$ is the minimum number of colors needed in a proper
coloring of a graph \cite{West}. In a proper coloring of a graph a
\emph{color class} is the independent set of all same colored
vertices of the graph. If $f$ is a proper coloring of $G$ with the
color classes $V_1$, $V_2$, ..., $V_{\ell}$ such that every vertex
in $V_i$ has color $i$, we simply write $f=(V_1,V_2,...,V_{\ell})$.

\emph{Domination} in graphs is now well studied in graph theory and the literature on this subject has been surveyed and detailed in the two
books by Haynes, Hedetniemi and Slater~\cite{HHSF,HHSD}.

Let $k$ be a positive integer. A subset $S$ of $V$ is a \emph{$k$-dominating set} of $G$, briefly $k$DS, if for every vertex $v\in V-S$, $\mid N_G(v)\cap S\mid \geq k$. The \emph{$k$-domination number} $\gamma _{k}(G)$ of $G$ is the minimum cardinality of a $k$-{dominating set} of $G$. For a graph to have a $k$-dominating set, its minimum degree must be at least $k-1$. A $k$-dominating set
of cardinality $\gamma _{k}(G)$ is called a $\gamma_{k}(G)$-\emph{set}. A \emph{$k$-domatic partition} of $G$ is a partition of $V(G)$ into $k$-dominating sets. The maximum number of $k$-dominating sets in a $k$-domatic partition of $G$ is called the \emph{$k$-domatic number} $d_k(G)$ of $G$. The $1$-domination and $1$-domatic numbers are known as \emph{domination number} and \emph{domatic number} and denoted by $\gamma(G)$ and $d(G)$, respectively. 

The $k$-domination number was first studied by Fink and Jacobson \cite{FJ0,FJ}, and Cockayne and Hedetniemi \cite{CH} introduced the concept of the domatic number of a graph. For more information on the domatic number and their variants, we refer the reader to the survey article by Zelinka \cite{BZ}. 

For positive integers $k$ and $\ell$, a $k$-dominating set $S$ of a graph $G$ is called a $(\ell,k)$-\emph{independent dominating set}, or briefly $(\ell,k)$-IDS, of $G$ if the subgraph induced by $S$ is a $\ell$-\emph{independent set}, that is, the maximum degree in the subgraph induced by $S$ is at most $\ell-1$. The $(\ell,k)$-independent domination number (resp. the $\ell$-independence number) of $G$ denoted by $i_{\ell,k}$ (resp. $\beta_{\ell}$) is the cardinality of the smallest $(\ell,k)$-IDS (resp. largest independent set) in $G$. Similarly, the concept of $(\ell,k)$-\emph{independent domatic number} $d_{(\ell,k)}(G)$ of a graph $G$ can be defined.
\vspace{0.2cm}

The following theorem provides a lower bound for the $k$-domination number of a graph in terms of the order and the maximum degree of the graph.
%-------------------------------------------------------------------------
\begin{thm} \emph{\cite{KV}}
\label{LB.gamma_k.Fink.1985} For any graph $G$ of order $n$,
\[
\gamma_{k}(G)\geq \frac{kn}{k + \Delta(G)}.
\]
\end{thm}
%-------------------------------------------------------------------------
As an immediately consequence of Theorem \ref{LB.gamma_k.Fink.1985},
we find the following proposition for Latin square graphs.

\begin{prop}
\label{LB.gamma_3(LSG)} For any Latin square graph $L_3(L,n)$,
$\gamma_{3}(L_3(L,n))\geq n.$
\end{prop}

%------------------------------ Sec. 2. Transversals  --------------------
\section{Transversals}
Every transversal in a Latin square $L$ gives a (1,3)-independent-dominating set of $L_3(L,n)$ with cardinality $n$. Next theorem presents a sufficient condition to guarantee its converse.
%---------------------------------------------------------------------
\begin{thm}
\label{Suf.Con.}
Let $S$ be a subset of the vertex set 
of the Latin square graph  $L_3(L,n)$ of a Latin square $L=(\ell_{i,j})$ of order $n$, and let $S'=\{\ell_{ij}~|~(i,j)\in S\}$. Then the following statements are equivalent.

\textbf{i}. $S$ is a $3$-dominating set of $L_3(L,n)$ with cardinality $n$.
\vspace{-0.01cm}

\textbf{ii}. $S$ is a $(1,3)$-independent dominating set of $L_3(L,n)$ with cardinality $n$.
\vspace{-0.01cm}

\textbf{iii}. $S'$ is a transversal in $L$.
\end{thm}

\begin{proof}
Obviously it is sufficient to prove that $i$ implies $iii$. For this aim, let $S$ be a $3$-dominating set of $L_3(L,n)$ with cardinality $n$. Since for any vertex $(i,j)$, 
\begin{equation*}
\begin{array}{lll}
 N((i,j))& = &\{i\}\times ([n]-\{j\})\\
& \cup &([n]-\{i\})\times \{j\}\\
& \cup & \{(p,q)~|~\ell_{pq}=\ell_{ij}\},
\end{array}
\end{equation*}\\
we conclude that for any $1\leq i,j \leq n$, 
$$|S\cap (\{i\}\times [n])|= |S\cap ([n]\times \{j\})|=1.$$
Also, we see that $|S\cap N((i,j))|=3$ for any vertex $(i,j)\notin S$. Without loss of generality, we may assume that $\ell_{1j}=j$, for $1\leq j\leq n$ and $|S\cap (\{1\}\times [n])|=\{(1,1)\}$. Then for any $2\leq j\leq n$, there exist an unique $2\leq i\leq n$ and an unique $(p,q)$ such that $p\neq 1$, $q\neq j$ and $\ell_{pq}=j$. Hence $S'=[n]$, that is, $S'$ is a transversal in $L$.
\end{proof}

%-------------------------------------------------------------------------
Theorem \ref{Suf.Con.} states that the maximum number of disjoint transversals in a Latin square is equivalent to the maximum number of disjoint $(1,3)$-independent dominating sets of cardinality $n$ in its Latin square graph, when the $3$-domination number of the graph is $n$. So we have the following theorem.
%-------------------------------------------------------------------------

\begin{thm}
\label{d_3(L_3(L,n))=tau (L)} Let $L$ be a Latin square of order
$n$. If $\gamma_3(L_3(L,n))=n$, then
\[
d_3(L_3(L,n))\geq d_{(1,3)}(L_3(L,n))=\tau(L).
\]
\end{thm}
%-------------------------------------------------------------------------
Since for any graph $G$ of order $n$, we have 
\begin{eqnarray}\label{(2.1)}
\label{d_k(G) =< n/ gamma_{k}(G)} d_k(G) \leq \frac
{n}{\gamma_{k}(G)}
\end{eqnarray}
from \cite{KV}, we may conclude the next result.
%-------------------------------------------------------------------------
\begin{thm}
\label{d_3(L_3(L,n)} For any Latin square $L$ of order $n$, $d_3(L_3(L,n)) \leq n$, and equality holds if and only if $\tau(L)=n$.
\end{thm}
%-------------------------------------------------------------------------

By two theorems from \cite{MM} and \cite{Wan}, the next theorem is obtained:

%-------------------------------------------------------------------------

\begin{thm}\label{ortho}
For any Latin square $L$ of order $n$, the following statements are equivalent.

\textbf{i}. $L$ has an orthogonal mate.
\vspace{-0.01cm}

\textbf{ii}. $\tau(L)=n$.
\vspace{-0.01cm}

\textbf{iii}. $\chi(L_3(L,n))=n$.
\end{thm}
%--------------------------------------------------

So, the next theorem can be obtained by Theorems \ref{d_3(L_3(L,n)}  and \ref{ortho}.

%------------------------------------------------------------------------------
\begin{thm}\label{ortho and d_3}
Let $L$ be a Latin square of order $n$. Then the following statements are equivalent.

\textbf{i}. $L$ has an orthogonal mate.
\vspace{-0.01cm}

\textbf{ii}. $\tau(L)=n$.
\vspace{-0.01cm}

\textbf{iii}. $\chi(L_3(L,n))=n$.
\vspace{-0.01cm}

\textbf{iv}. $d_3(L_3(L,n))=n$.
\end{thm}
%------------------------------------------------------

We know from \cite{Wan} that for any group $G$ of order $n$, $\tau(L_{G})\geq 1$ 
implies that  $\tau(L_{G})=n$. Now since the diagonal of the table of an abelian 
group $G$ of odd order $n$ is a transversal, hence $\tau(L_{G})=n$. Also, we know 
that every Latin square of even order has an even number of transversals \cite{K}, 
and for any group $G$ of order $2n$ which has an unique element of order $2$, 
$\tau(L_G)=0$ \cite{Euler}. The next theorem states that every $2$-step type Latin square of order $2^n\geq 4$ has a decomposition into disjoint transversals.

%----------------------------------------------------------------------------------------------------------------------------------
\begin{thm}\label{tau=2^n}
For any $2$-step type Latin square $L$ of order $2^n\geq 4$, $\tau(L)=2^n$.
\end{thm}
\begin{proof}
Our proof is presented by induction on $n\geq 2$. For $n=2$, Figure \ref{fig:3} shows the only $2$-step type Latin square (up to isomorphism) of order $4$, and its four transversals. 

%------------------------------------------------------------
\begin{figure}[h]
    \centering
    \includegraphics[width=25mm, height=25mm]{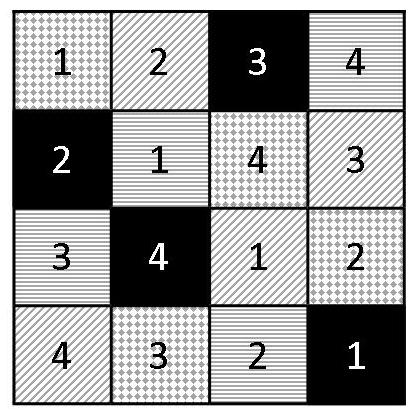}
    %\vspace{-2cm}
    \caption{$4$ disjoint transversals of a $2$-step type Latin square of order $4$}
    \label{fig:3}
\end{figure}
%---------
So, we assume that $L$ is a $2$-step type Latin square
of order $2^{n+1}>4$, and every $2$-step type Latin square of order $2^n$ has $2^n$ disjoint transversals.
We will prove $\tau(L)=2^{n+1}$.
Since $L$ is also a $2^n$-step type Latin square, it can be represented by a matrix of $2^n\times 2^n$ blocks
$A_{ij}$ as follows 
\begin{equation*}
L=\left(\begin{array}{cc}
A_{11} & A_{12}\\
A_{21} & A_{22}
\end{array}\right)
\end{equation*}
%----------------
where $A_{11}=A_{22}$ is a $2$-step type Latin square on the symbol set $[n]=\{1,2,\cdots,2^n\}$,
and $A_{12}=A_{21}$ is a $2$-step type Latin square on the symbol set $2^n+[n]=\{2^n+i~|~i\in [n]\}$,
and the $(i,j)$-entry in $A_{12}$ is the sum of the $(i,j)$-entry in $A_{11}$ and $2^n$. In fact,
for the bijection $\sigma(i)=i+2^n$ from $[n]$ to $2^n+[n]$, $A_{12}=\sigma(A_{11})$. Let also 
%-------------
\begin{equation*}
\begin{array}{cc}
A_{11}=(A_{11}',A_{11}'') & A_{12}=(A_{12}',A_{12}'')\\
A_{21}=(A_{21}',A_{21}'') & A_{22}=(A_{22}',A_{22}'')
\end{array}
\end{equation*}
%-------------
where $A_{ij}',A_{ij}''$ are $2^{n-1}$-by-$2^n$ Latin rectangles. Now let $T$ be a transversal in $A:=A_{11}=A_{22}$. $T_{11}$ and $T_{22}$
are the name of $T$ in $A_{11}$ and $A_{22}$, respectively. Let $T_{11}=(T_{11}',T_{11}'')$ and 
$T_{22}=(T_{22}',T_{22}'')$, where $T_{ii}'=T_{ii}\cap A_{ii}'$ and $T_{ii}''=T_{ii}\cap A_{ii}''$. Then $(\sigma (T_{11}''),\sigma (T_{11}'))$ is a transversal in $A_{12}=A_{21}$. 
Let $T_{12}$ and $T_{21}$ be the name of $(\sigma (T_{11}''),\sigma (T_{11}'))$ in 
$A_{12}$ and $A_{21}$, respectively. Since 
%------------
\begin{equation*}
\begin{array}{cc}
T_1=T_{11}'\cup T_{22}''\cup T_{12}''\cup T_{21}' & \\
T_2=T_{11}''\cup T_{22}'\cup T_{12}'\cup T_{21}''
\end{array}
\end{equation*}
%------------
are two disjoint transversals in $L$, and $\tau(A)=2^n$, we obtain $\tau(L)=2^{n+1}$. 
See Figures \ref{fig:4}, \ref{fig:5} and \ref{fig:6} for example.
%------------------------------------------------------------
\begin{figure}[h]
    \centering
    \includegraphics[width=23mm, height=23mm]{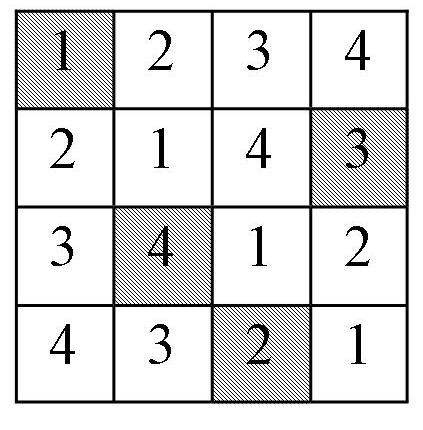}
    %\vspace{-2cm}
    \caption{$A=A_{11}=A_{22}$, and a transversal}
    \label{fig:4}
\end{figure}
%------------------------------------------------------------
\begin{figure}[h]
    \centering
    \includegraphics[width=52mm, height=47mm]{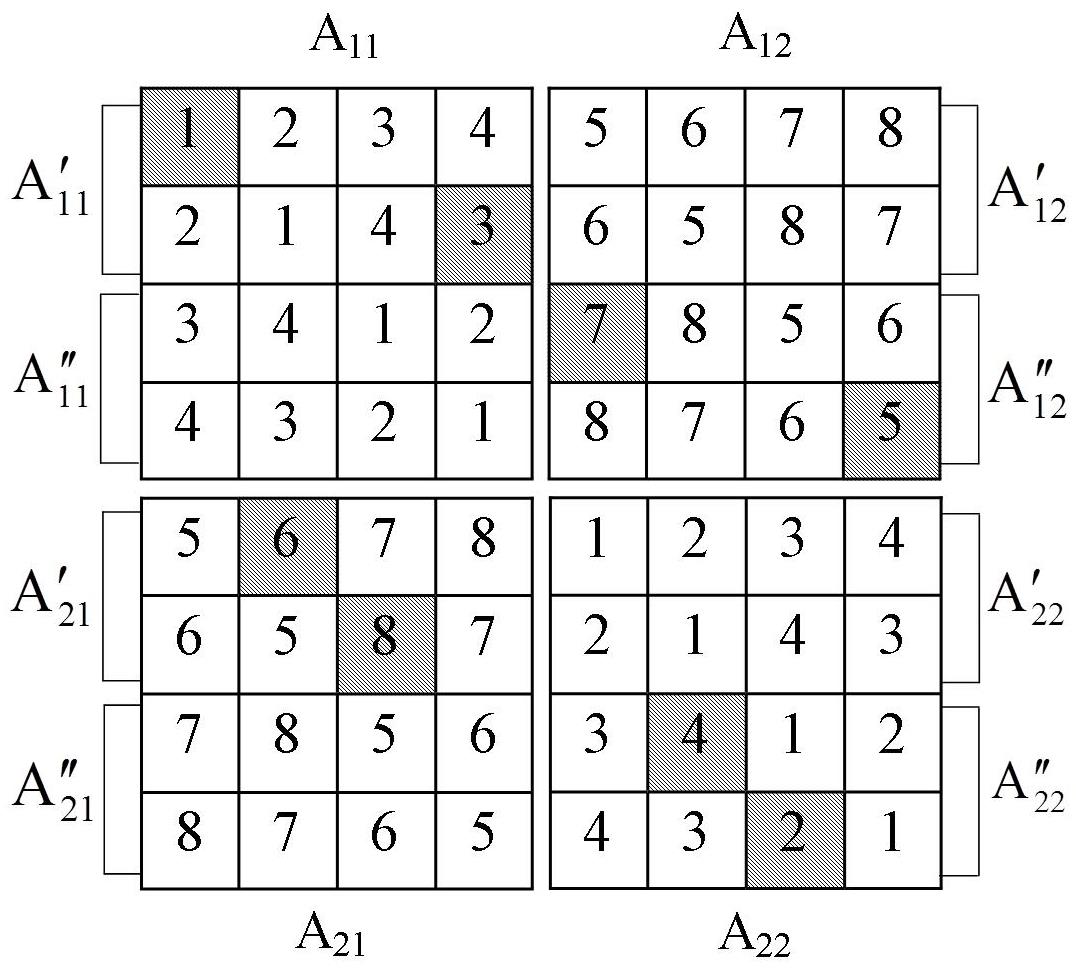}
    %\vspace{-2cm}
    \caption{Partition of a $2$-step type Latin square of order $2^{3}$}
    \label{fig:5}
\end{figure}
%------------------------------------------------------------
\begin{figure}[h]
    \centering
    \includegraphics[width=42mm, height=42mm]{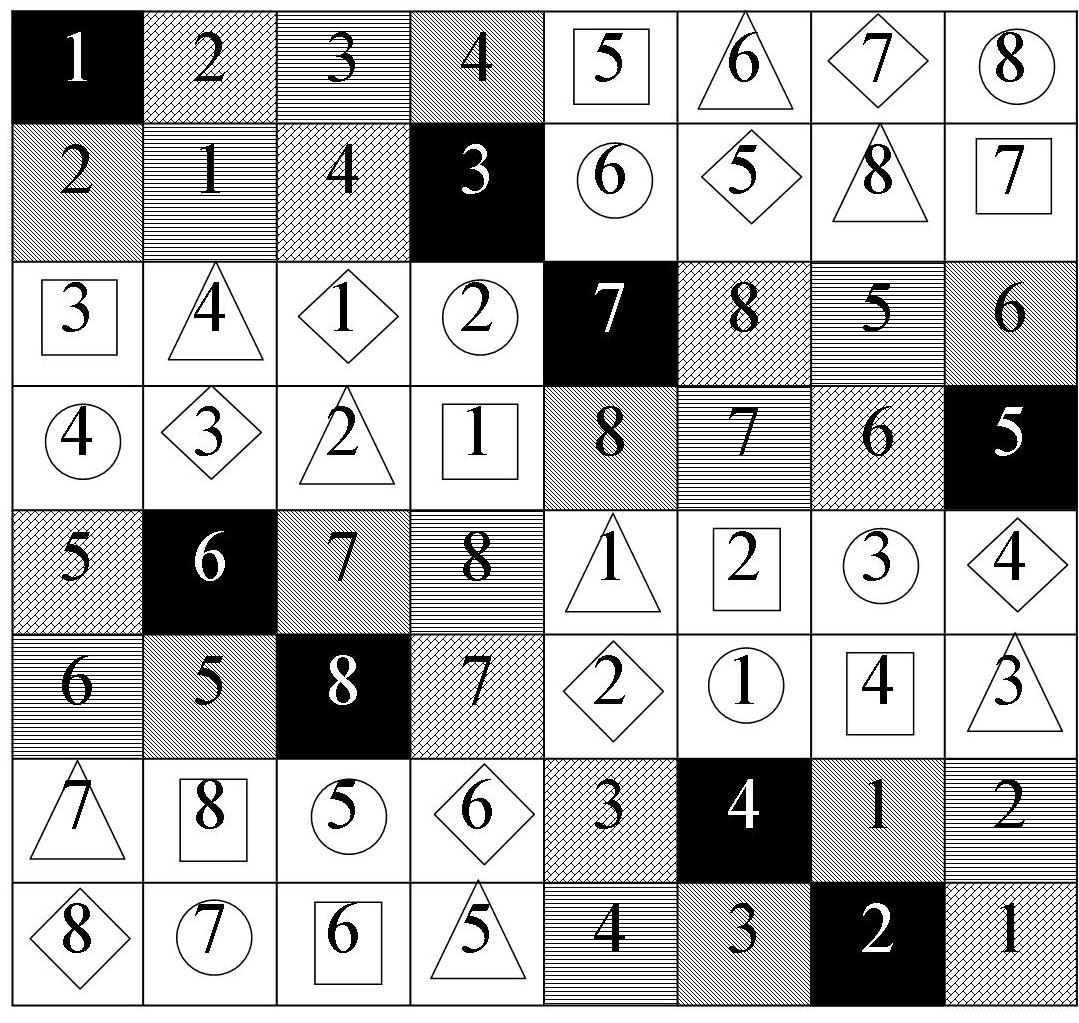}
    %\vspace{-2cm}
    \caption{A $2$-step type Latin square of order $2^{3}$ and its transversals}
    \label{fig:6}
\end{figure}
%------------------------------------------------------------
\end{proof}

%------------------------------------------------------------
As a consequence of Theorems \ref{tau=2^n} and \ref{ortho and d_3}, next result is appearanced.

\begin{cor}\label{tau=2^n}
For any $2$-step type Latin square $L$ of order $2^n\geq 4$, $d_3(L_3(L_G,2^n))=2^n$.
\end{cor}

%-------------------------------Sec. 3. Quasi-Transversals -----------------------------------
%-------------------------------Sec. 3. Quasi-Transversals -----------------------------------

\section{Quasi-transversals}
%---------------------------------------------------------------------------------------
Using proven theorems about transversals of Latin squares, there are many Latin squares which have no transversal, but they have a set of entries of order $n+1$ which includes exactly two entry from one row (and one column)  and exactly one entry from each other row (and column) and also all symbols in the set except one occure one time. This guides our to the next definition.
%---------------------------------------------------------------------------------
\begin{defn}
\emph{A} quasi-transversal \emph{in a Latin square of order $n$ is a set of entries of order $n+1$ which
includes exactly two entry from one row (and one column)  and one entry from each other row (and column) and all symbols in the set, except one, occure one time. The maximum number of disjoint quasi-transversals in a Latin square $L$ is called its} quasi-transversal number \emph{and denoted by $\tau^{qua}(L)$.}
\end{defn}
%-----------------------------------------------------------------------
It is obviouse that for any a Latin square $L=(\ell_{i,j})$ of order $n$, $\tau^{qua}(L)\leq \lfloor\frac{n^2}{n+1}\rfloor$. Next proposition states when a quasi-transversal in a Latin square is a 3-dominating set in its Latin square graph. 

%-----------------------------------------------------------------------
%\begin{color}{red}
\begin{prop}
\label{quasi-transversal=3DN n+1}
Every quasi-transversal in a Latin square of order $n\geq 3$ is corresponded with a 3-dominating set of order $n+1$ in its Latin square graph.
\end{prop}

\begin{proof}
For $n=3$, there is only one isotopy class of Latin squares of order $3$, and it has a quasi-transversal (see Figure \ref{isotopy class, n=3}).
%-----------------------------------------
\begin{figure}[h]
    \centering
    \includegraphics[width=20mm, height=20mm]{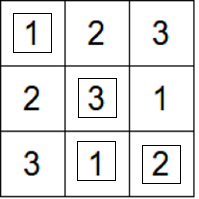}
    %\vspace{-2cm}
    \caption{The only one isotopy class of Latin squares of order $3$}
    \label{isotopy class, n=3}
\end{figure}
%--------------

So, let $n\geq 4$. It is clear that every quasi-transversal in a Latin square of order $n\geq 3$ is a 3-dominating set of order $n+1$ in its Latin square graph. Let $T$ be a $3$DS in $L_3(L,n)$ of cardinality $n+1\geq 5$, and let
\begin{equation*}
\begin{array}{ccc}
R_i &=&\{(i,j)  \in V(L_3(L,n))~|~(i,j)\in T \},\\
C_i &=&\{(j,i)  \in V(L_3(L,n))~|~(j,i)\in T \},\\
S_i &=&\{(p,q) \in V(L_3(L,n))~|~\ell_{pq}=i \}
\end{array}
\end{equation*}
be three sets of the cardinalities $r_i$, $c_i$ and $s_i$, respectively. Then
\begin{equation*}
\begin{array}{ccc}
T&=& R_1\cup R_2\cup \cdots R_n\\
  &=& C_1\cup C_2\cup \cdots C_n\\
  &=&S_1\cup S_2\cup \cdots S_n,
\end{array}
\end{equation*}  
and so 
\[
\sum_{i=1}^{n}r_i=\sum_{i=1}^{n}c_i=\sum_{i=1}^{n}s_i=n+1.
\]
We will show that the set $T'=\{\ell_{ij}~|~(i,j)\in T\}$ is a quasi-transversal in $L$. For this aim, it is sufficient to prove $r_i\neq 0$ for any $1\leq i \leq n$, by the concept of orthogonal mate.
Let $I=\{i~|~r_i=0\}$ be a set of cardinality $\ell\geq 1$. Without loss of generality, assume that $|\{i~|~c_i=0\}|\geq \ell$ and $|\{i~|~s_i=0\}|\geq \ell$. Let $J=\{i~|~c_i=0\}$. This implies that 
if $\alpha_{ij}$ is the symbol in the intersection of the $i$-th row and the $j$-th column for $i\in I$ and $j\in J$, then $s_{\alpha_{ij}}\geq 3$. Let $K=\{t~|~ \alpha_{ij}\in T \mbox{ and } \alpha_{ij} \mbox{ lies in $t$-th column, for } i\in I, ~j\in J\}$. Trivally, $|K|\geq \ell$. So, if $K'=\{i~|~1\leq i \leq n\}-J\cup K$, then 
\begin{equation*}
\begin{array}{lll}
n+1 &=     & |T|\\
       &=     & \sum_{i\in K'}c_i + \sum_{i\in K}c_i \\
       &\geq & (n-2\ell)+3\ell \\
       &=     & n+\ell ,
\end{array}
\end{equation*}
which implies $\ell=1$. So, without loss of generality, consider $r_1=c_1=s_{\alpha_1}=0$, for some $\alpha_1\neq 1$, $\ell_{1,j}=\ell_{j,1}=j$ for $1\leq j\leq n$ and  $c_1\leq c_2\leq \cdots \leq c_n$. Since $c_j\leq 3$ for $2\leq j\leq n$, we obtain $(c_2,c_3,\cdots,c_n)=(1,1,\cdots,1,3)$ or
$(c_2,c_3,\cdots,c_n)=(1,\cdots,1,2,2)$.
This gives that $r_i\leq 2$ for each $i$. Since at least three vertices from $T$ is needed to $3$-tuple dominating of
the vertex $(\alpha_1,1)$, the condition $r_1=c_1=s_{\alpha_1}=0$ implies that $r_{\alpha_1}\geq 3$, a contradiction. Therefore $T'$ is a quasi-transversal in $L$ with cardinality $n+1$.
\end{proof}
%\end{color}
%-----------------------------------------------------------------------
The next theorem is presented a family of Latin squares of order $n$ which has no transversal but has a quasi-transversal, which implies that the 3-domination number of its Latin square graph is $n+1$. First we recall a theorem from \cite{WCS}.

%--------------------------------------------------------------------------

\begin{thm}\label{n even k odd}
\emph{\cite{WCS}}
Suppose that $q$ and $k$ are odd integers and $m$ is even. No $q$-step type
Latin square of order $mq$ possesses a $k$-transversal.
\end{thm}
%---------------------------------------------------------------------------------------

\begin{thm}\label{gamma=n+1}
For any $q$-step type Latin square $L$ of even order $n=mq$ where the $q\times q$ sublatin squares are cyclic, $m$ is even and $q$ is odd,
 $\gamma_3(L_3(L,n))=n+1$.
\end{thm}

\begin{proof}
Since $\gamma_{3}(L_3(L,n))\geq n+1$, by Theorems \ref{Suf.Con.} and \ref{n even k odd}, it is sufficient to prove $\gamma_{3}(L_3(L,n))\leq n+1$.

\textbf{Case 1}. $q=1$. Since the set 
\begin{equation*}
\begin{array}{lll}
S &  =  &\{(i,i)\mid 1\leq i\leq \frac{n}{2}\}\\
& \cup &\{(i,i+1) \mid \frac{n}{2}+1\leq i\leq n\}\\
& \cup &
\{(\frac{n}{2}+1,\frac{n}{2}+1)\}
\end{array}
\end{equation*}\\
is a $3$-dominating set of $L_3(L,n)$ of the cardinality $n+1$, we obtain $\gamma_{3}(L_3(L,n))=n+1$.
 
\textbf{Case 2}. $q\neq 1$. Since the set
\begin{equation*}
\begin{array}{lll}
S &  =  &\{(qj+2i +1,qj+2i+1)\mid 0\leq i \leq \frac{q-1}{2},0\leq j\leq \frac{m}{2}-1\}\\
& \cup &\{(qj+2i +1+\frac{n}{2},q(j+1)+2i +1+\frac{n}{2})\mid 0\leq i \leq \frac{q-1}{2},0\leq j\leq \frac{m}{2}-2\}\\
& \cup &\{(qj+2i +2,q(j+1)+2i +2\mid 0\leq i \leq \frac{q-3}{2},0\leq j\leq \frac{m}{2}-2\}\\
& \cup &\{(qj+2i +2+\frac{n}{2},qj+2i +2+\frac{n}{2})\mid 0\leq i \leq \frac{q-3}{2},0\leq j\leq \frac{m}{2}-1\}\\
& \cup &\{(2i+2-q+\frac{n}{2},2i +1+\frac{n}{2})\mid 0\leq i \leq \frac{q-3}{2}\}\\
& \cup &\{(2i +1-q+n,2i +4)\mid 0\leq i \leq \frac{q-5}{2}\}\\
& \cup &\{(\frac{n}{2}-1,\frac{n}{2}+q),(n-2,2),(n,1)\}
\end{array}
\end{equation*}\\
is a $3$DS of $L_3(L,n)$ of the cardinality $n+1$, we obtain $\gamma_{3}(L_3(L,n))=n+1$. For example, Figure \ref{Latin square with (n,m,q)=(36,4,9)} shows a 9-step Latin square of order 36 and a quasi-transversal.
%----------------------------------------------------
\begin{figure}[h]
    \centering
    \includegraphics[width=120mm, height=110mm]{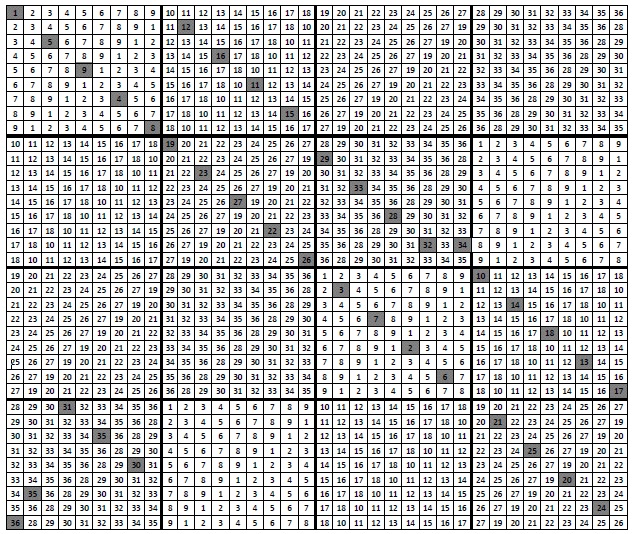}
    %\vspace{-2cm}
    \caption{A 3DS of a $9$-step type Latin square with $(n,m,q)=(36,4,9)$}
    \label{Latin square with (n,m,q)=(36,4,9)}
\end{figure}

\end{proof}

%---------------------------------------------------------------------------------
\begin{cor}\label{tau_k L_G, G cyclic of order 2n}
For any cyclic group $G$ of even order $n$, $d_3(L_3(L_{G},n))=\lfloor \frac{n^2}{n+1}\rfloor$.
\end{cor}

\begin{proof}
By Remark \ref{1-step} and Theorem \ref{gamma=n+1}, $\gamma_{3}(L_3(L_{G},n))= n+1$. 
Now define
\begin{equation*}
{S}_{j}=\left\{
\begin{array}{ll}
T_j \cup \{(j+\frac{n}{2},j+\frac{n}{2})\}         & \mbox{if }1\leq j\leq \frac{n}{2}, \\
T_j \cup \{(j+\frac{n}{2}+1,j+\frac{n}{2})\}       & \mbox{if }\frac{n}{2}< j < n-1, \\
T_j \cup \{(\frac{n}{2},\frac{n}{2}-1),(1,\frac{n}{2})\} & \mbox{if } j=n-1,
\end{array}
\right.
\end{equation*}
where $1\leq j \leq {n}-1$, and $T_j=\{(i,i+j-1),(i+\frac{n}{2},i+\frac{n}{2}+j)~|~1\leq i \leq \frac{n}{2} \}$. Since the sets $S_j$ are $n-1$ disjoint 
$3$-dominating sets of $L_3(L_{G},n)$, so $d_3(L_3(L_{G},n))\geq n-1$. On the other hand, $d_3(L_3(L_{G},n))\leq \frac{n^2}{n+1}$ implies $d_3(L_3(L_{G},n))\leq n-1$. Hence $d_3(L_3(L_{G},n))=\lfloor \frac{n^2}{n+1}\rfloor$. 
Figure \ref{fig:1} shows that $9$ disjoint $3$DSs of a cyclic Latin squares of order $10$.
%--------------------------------------------------------------------
\begin{figure}[h]
    \centering
    \includegraphics[width=53mm, height=53mm]{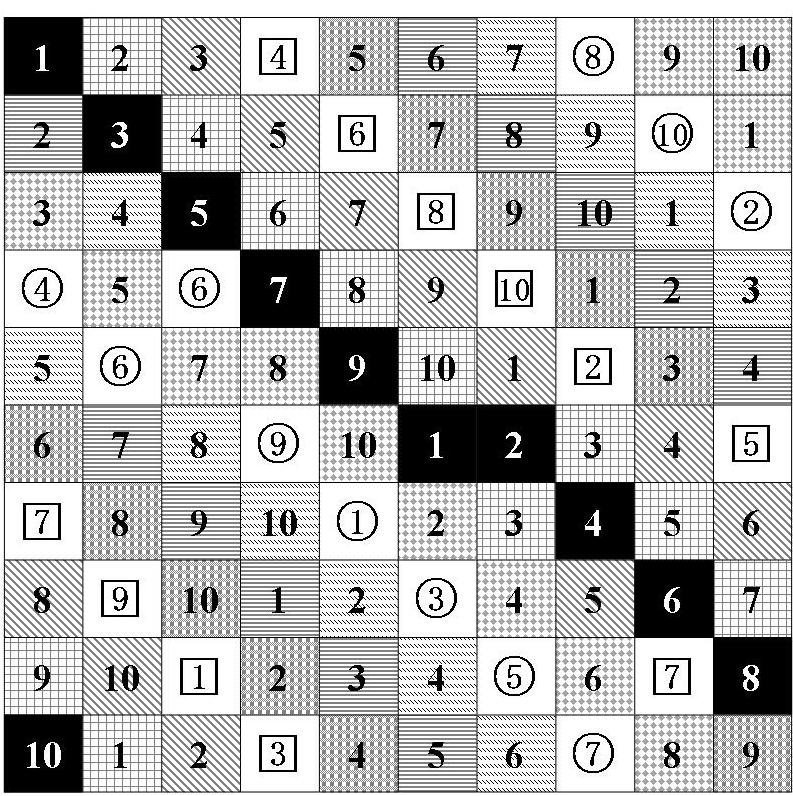}
    %\vspace{-2cm}
    \caption{$9$ disjoint $3$DSs of a cyclic Latin squares of order $10$}
    \label{fig:1}
\end{figure}
\end{proof}

%-------------------------------------------------------------------------------
As a consequence of Theorems \ref{quasi-transversal=3DN n+1}, \ref{n even k odd} and \ref{gamma=n+1}, the next result presents a family of Latin squares which has no transversal but has a quasi-transversal. 
%, which implies that the 3-domination number of its Latin square graph is $n+1$.First we recall a theorem from \cite{WCS}.

%--------------------------------------------------------------------------
\begin{cor}\label{$q$-step}
Any $q$-step type Latin square $L$ of even order $n=mq$ where the $q\times q$ sublatin squares are cyclic, $m$ is even and $q$ is odd, has no transversal but has a quasi-transversal.
\end{cor}
Corollary \ref{$q$-step} implies that for even $n$ the Cayley table of the addition groups $\mathbb{Z}_n$, which is an $1$-step 
Latin square, has a quasi-transversal while it has no transversal.

\vspace{0.2cm}

%---------------------------------------------
The following observation states the relationship between transversals, quasi-transversals and near-transversals in a Latin square. 

%------------------------------------------------------------------------------
\begin{obs}
\label{QT=NT}
Let $L$ be a Latin square oforder $n\geq 3$. Then\\
\textbf{i}. every transversal in $L$ is contined in the quasi-transversal which is obtained by adding one symbol to the transversal,\\
\textbf{ii}. every quasi-transversal in $L$ contains the near-transversal which is obtained by deleting the repeated symbol in the quasi-transversal,\\
\textbf{iii}. every near-transversal in $L$ is contained in the quasi-transversal which is obtained by adding the missing symbol two times, one time in the empty row and one time in the empty column,\\
\textbf{iv}. every quasi-transversal in $L$ contains a transversal if and only if the near-transversal is obtained from it is completable.
\end{obs}

%------------------------------------------------------------------------------
The following conjecture has been attributed to Brualdi (see \cite{Keedwell}, p.103) and Stein \cite{Stein} and, in \cite{EHNS}, to Ryser. For generalisations of it, in terms of hypergraphs, see \cite{AB}.

%------------------------------------------------------------
\begin{conj}
\label{Latin has NT} Every Latin square has a near-transversal.
\end{conj}
%------------------------------------------------------------
By considering Observation \ref{QT=NT}, the conjecture \ref{Latin has NT} is equivalent to the following our conjecture.
%------------------------------------------------------------
\begin{conj}
\label{Latin square has QT} Every Latin square has a quasi-transversal.
\end{conj}

%-------------Some family of Latin squares which has a 2-transversal -----------------------------------
\section{The Rodney's conjecture}

%----------------------------------------
Rodney in \cite{Rodney} has given the next conjecture. 

%----------------------------------------------------------
\begin{conj} 
Every Latin square has a 2-transversal.
\end{conj}
%------------------------------------------------------
Here, we prove the truth of the Rodney's conjecture for a family of Latin squres.

\begin{thm}
Every $q$-step type Latin square $L$ of even order $n=mq$, where the $q\times q$ sublatin squares are cyclic, $m$ is even and $q$ is odd, has a $2$-transversal.
\end{thm}

\begin{proof}
\textbf{Case 1}. $q=1$. The set 
\begin{equation*}
\begin{array}{lll}
S &  =  &\{\ell_{i,i}\mid 1\leq i\leq \frac{n}{2}\}\\
& \cup &\{\ell_{i,i+1} \mid \frac{n}{2}+1\leq i\leq n\}\\
& \cup &
\{\ell_{\frac{n}{2}+1,\frac{n}{2}+1}\}
\end{array}
\end{equation*}\\
is a quasi-transversal in $L$ with $\ell_{1,1}=\ell_{\frac{n}{2}+1,\frac{n}{2}+1}=1$, and the set
\begin{equation*}
\begin{array}{lll}
S' &  =  &\{\ell_{i,i+1}\mid 1\leq i\leq \frac{n}{2}\}\\
& \cup &\{\ell_{i,i} \mid \frac{n}{2}+2\leq i\leq n\}
\end{array}
\end{equation*}\\
is a near-transversal in $L$ with the missing entry 1 in the $(n/2+1)$-th row and the $(n/2+1)$-th column. Then $S\cup S'$ is a $2$-transversal of $L$, because $S\cap S'=\emptyset$.
 %------------------
 
\textbf{Case 2}. $q\neq 1$ and $m=2$. The set
 \begin{equation*}
\begin{array}{lll}
S &  =  &\{\ell_{2i +1,2i +1}\mid 0\leq i \leq {q-1}\}\\
& \cup &\{(\ell_{2i +2,2i+3+\frac{n}{2}}\mid 0\leq i \leq \frac{q-3}{2}\}\\
& \cup &\{\ell_{2i +1+\frac{n}{2},2i +2}\mid 0\leq i \leq \frac{q-3}{2}\}\\
& \cup &\{\ell_{\frac{n}{2},{\frac{n}{2}}+1},\ell_{n,1}\}
\end{array}
\end{equation*}\\
is a quasi-transversal in $L$ with $\ell_{n,1}=\ell_{\frac{n}{2},\frac{n}{2}+1}=n$, and the set
\begin{equation*}
\begin{array}{lll}
S' &  =  &\{\ell_{2i +2,2i +2}\mid 0\leq i \leq {q-1}\}\\
& \cup &\{\ell_{2i+1,2i+2+\frac{n}{2}}\mid 0\leq i \leq \frac{q-3}{2}\}\\
& \cup &\{\ell_{2i +2+\frac{n}{2},2i +3}\mid 0\leq i \leq \frac{q-3}{2}\}\\
\end{array}
\end{equation*}\\
is a near-transversal in $L$ with the missing entry $n$ in the $n/2$-th row and the first column. Then $S\cup S'$ is a $2$-transversal of $L$, because $S\cap S'=\emptyset$.
%-------------------

\textbf{Case 3}. $q\neq 1$ and $m\neq 2$. Since the set
\begin{equation*}
\begin{array}{lll}
S &  =  &\{\ell_{qj+2i +1,qj+2i +1}\mid 0\leq i \leq \frac{q-1}{2},0\leq j\leq \frac{m}{2}-1\}\\
& \cup &\{\ell_{qj+2i +1+\frac{n}{2},q(j+1)+2i +1+\frac{n}{2}}\mid 0\leq i \leq \frac{q-1}{2},0\leq j\leq \frac{m}{2}-2\}\\
& \cup &\{\ell_{qj+2i +2,q(j+1)+2i +2}\mid 0\leq i \leq \frac{q-3}{2},0\leq j\leq \frac{m}{2}-2\}\\
& \cup &\{\ell_{qj+2i +2+\frac{n}{2},qj+2i +2+\frac{n}{2}}\mid 0\leq i \leq \frac{q-3}{2},0\leq j\leq \frac{m}{2}-1\}\\
& \cup &\{\ell_{2i +2+\frac{n}{2}-q,2i +1+\frac{n}{2}}\mid 0\leq i \leq \frac{q-3}{2}\}\\
& \cup & \{\ell_{n-2i ,q-(2i +1)}\mid 0\leq i \leq \frac{q-3}{2}\}\\
& \cup &\{\ell_{\frac{n}{2}+1-q,\frac{n}{2}+q},\ell_{n-q+1,q}\}
\end{array}
\end{equation*}\\
is a quasi-transversal in $L$ with $\ell_{\frac{n}{2}+1-q,\frac{n}{2}+q}=\ell_{n-q+1,q}=n$, and the set
\begin{equation*}
\begin{array}{lll}
S' &  =   &\{\ell_{qj+2i +2,qj+2i +2}\mid 0\leq t \leq \frac{q-3}{2},0\leq j\leq \frac{m}{2}-1\}\\
& \cup &\{\ell_{qj+2i +1+\frac{n}{2},qj+2i +1+\frac{n}{2}}\mid 0\leq i\leq \frac{q-1}{2},0\leq j\leq \frac{m}{2}-1\}\\
& \cup &\{\ell_{qj+2i +1,q(j+1)+2i +1}\mid 0\leq i \leq \frac{q-1}{2},0\leq j\leq \frac{m}{2}-2\}\\
& \cup &\{\ell_{qj+2i +2+\frac{n}{2},q(j+1)+2i +2+\frac{n}{2}}\mid 0\leq i \leq \frac{q-3}{2},0\leq j\leq \frac{m}{2}-2\}\\
& \cup &\{\ell_{2i +3-q+\frac{n}{2},2i+2+\frac{n}{2}}\mid 0\leq i \leq \frac{q-3}{2}\}\\
& \cup & \{\ell_{n-(2i +1) ,q-(2i +2)}\mid 0\leq i \leq \frac{q-3}{2}\}\
\end{array}
\end{equation*}\\
is a near-transversal in $L$ with the missing entry $n$ in the ($\frac{n}{2}+1-q$)-th row and the $q$-th column. Then $S\cup S'$ is a $2$-transversal of $L$, because $S\cap S'=\emptyset$.
\end{proof}
 
%---------------------------------------------------------- References ----------------------------------------------------------


\begin{thebibliography}{20}

\bibitem{AB} R. Aharoni and E. Berger, Rainbow matchings in $r$-partite $r$-graphs, {\em Electron. J. Combin.} {\bf 16} (2009), R119.

\bibitem{K} K. Balasubramanian, On transversals in latin squares, {\em Lin. Alg. Appl.} {\bf 131} (1990) 125--129. 

%\bibitem{CS}  E. W. Clark, S. Suen, An inequality related to Vizing's Conjecture, {\em Electron. J. Combin.} {\bf 7} No.1 (2000) Note 4, 3 pp. (electronic).

\bibitem{CH}  E. J. Cockayne and S. T. Hedetniemi, Towards a theory of domination in graphs, {\em Networks.} {\bf 7} (1977) 247--261.

%\bibitem{CD}  C. J. Colbourn and J. H. Dinitz, {\em Handbook of Combinatorial Designs Second Edition,} 2007.

\bibitem{Keedwell} J. Denes, A.D. Keedwell, Latin Squares and their Applications, {\em Akad\'{e}miai
Kiad\'{o}, Budapest,} 1974.

%\bibitem{Finney} D. J. Finney, Some orthogonal properties of the $4\times 4$ and $6\times 6$ Latin squares. {\em Ann. Eugenics.} {\bf 12} (1945) 213--219.

\bibitem{EHNS} P. Erd˝os, D.R. Hickerson, D.A. Norton and S.K. Stein, Unsolved problems:
Has every latin square of order n a partial latin transversal of size n − 1?
{\em Amer. Math. Monthly,} {\bf 95} (1988), 428--430.

\bibitem{Euler} L. Euler, Recherches sur une nouvelle esp\'{e}ce de quarr\'{e}s magiques, {\em Verh.
Zeeuwsch. Gennot. Weten. Vliss.} {\bf 9} (1782), 85--239. 

%\bibitem{HH}  F. Harary, T. W. Haynes, Double domination in graphs, {\em Ars Combin.} {\bf 55} (2000) 201--213.

\bibitem{FJ0}  J. F. Fink and M. S. Jacobson, $n$-domination in graphs, {\em Graph Theory with Applications to Algorithms and Computer Science,} John Wiley and Sons, New York (1985) 282--300.

\bibitem{FJ}  J. F. Fink and M. S. Jacobson, On $n$-domination, $n$-dependence and forbidden subgraphs, {\em Graph Theory with Applications to Algorithms and Computer Science,} John Wiley and Sons, New York (1985) 301--311.

\bibitem{HHSF}  T. W. Haynes, S. T. Hedetniemi, P. J. Slater, {\em Fundamentals of Domination in Graphs,} Marcel Dekker, New York, 1998.

\bibitem{HHSD}  T. W. Haynes, S. T. Hedetniemi, P. J. Slater, {\em Domination in Graphs: Advanced Topics,} Marcel Dekker, New York, 1998.

\bibitem{KV}  K. Kammerling, L. Volkmann, The $k$ -domatic number of a graph, {\em Czechoslovak Mathematical Journal,} {\bf 59} (2009), 539--550.

\bibitem{MM}  E. S. Mahmoodian, M. Mortezaeifar, {\em Thesis: The Coloring of Block Design and Latin Square,} Sharif University and Technology, Iran, (2009).

\bibitem{Mailet} E. Maillet,  Sur les carr\'{e}es latins d'Euler, {\em C. R. Assoc. France Av. Sci.} {\bf 23} (1894), part 2, 244--252.

\bibitem{Mann} H. B. Mann, On orthogonal latin squares, {\em Bull. Amer. Math. Soc.} {\bf 50} (1944), 249--257. 

\bibitem{Parker6} E. T. Parker, Nonextendibility conditions on mutually orthogonal latin squares,
 {\em Proc. Amer. Math. Soc.} {\bf 13} (1962), 219--221.

\bibitem{Parker9} E. T. Parker, Pathological latin squares, {\em Proc. Sympos. Pure Math,} Amer. Math. Soc.
{\bf 19} (1971), 177--181. 

\bibitem{Rodney} P. Rodney, The existence of interval-balanced tournament designs, {\em J. Combin.
Math. Combin. Comput.} {\bf 19} (1995) 161--170.

%\bibitem{R} H. J. Ryser, Neuere Probleme der Kombinatorik, {\em Vortrage \"{u}ber Kombinatorik Oberwolfach}, {\bf 24--29} (1967), 69--91.

\bibitem{Stein} S. K. Stein, Transversals of Latin squares and their generalizations, {\em Pacific J.
Math.} {\bf 59} (1975), 567--575.

%\bibitem{VW}  J. H. van Lint, R. M. Wilson, {\em A Course in Combinatorics,} Cambridge University Press, New York, 2001.

\bibitem{Wan}  Ian M. Wanless, Transversals in Latin Squares, {\em Quasigroups and Related Systems,} {\bf 15} (2007) 169--190.

\bibitem{WCS}  Ian M. Wanless, C. Church, St. Aldates, A Generalisation of Transversals for Latin Squares, 
{\em The Electronic Journal of Combinatorics,} {\bf 9} (2002).

\bibitem{West} D. B. West, Introduction to Graph Theory, {\em Second edition: Prentice Hall,} 2001.

\bibitem{BZ} B. Zelinka, Domatic numbers of graphs and their variants, A survey Domination in Graphs {\em Advanced Topics. Marcel Dekker}, New York, (1998), pp. 351--374.



%\bibitem{HK8}  M. A. Henning, A. P. Kazemi, $k$-tuple total domination in graphs, {\em Discrete Applied Mathematics} {\bf 158} (2010)1006-1011.

%\bibitem{HR9}  M. A. Henning, D. F. Rall, On the total domination number of Cartesian products of graphs, {\em Graphs and Comibinatorics}{\bf 21} (2005) 63-69.

%\bibitem{KP}  A. P. Kazemi, B. Pahlavsay, $k$-tuple total domination in Supergeneralized Petersen graphs, {\em Communications in Mathematicsand Applications} {\bf 2} No.1 (2011) 21-30.

%\bibitem{Viz}  V. G. Vizing, Some unsolved problems in graph theory, {\em Usp. Mat. Nauk} {\bf 23} 144 (1968) 117-134.
\end{thebibliography}
\end{document}